\documentclass[12pt,times,reqno]{amsart}
\usepackage[english]{babel}  \usepackage{latexsym,amsfonts} \usepackage{amsmath, amssymb, amsthm} \usepackage{mathrsfs}
\usepackage{enumerate} \usepackage{lmodern} \usepackage{mathpazo}   \usepackage{euscript} 

\usepackage[top=2cm,bottom=2cm,left=2.4cm,right=2.4cm]{geometry}
\numberwithin{equation}{section}  \makeatletter\@addtoreset{equation}{section}
\newtheorem {theorem}{Theorem}[section]        
  \newtheorem {lemma}[theorem]{Lemma}
\newtheorem {definition}[theorem]{Definition}   \newtheorem {corollary}[theorem]{Corollary}     \newtheorem {remark}[theorem]{Remark}
\newtheorem {proposition}[theorem]{Proposition}       
         
\newcommand{\C}{\mathbb C}    \newcommand{\R}{\mathbb R}   \newcommand{\Z}{\mathbb Z} 	
\newcommand{\norm}[1]{\left\Vert#1\right\Vert}

\newcommand{\Lat}{L}

\newcommand{\Oschi}{\mathcal{O}^{\nu}_{\Gamma,\chi}(\C) }
\newcommand{\Osdchi}{\mathcal{F}^{2,\nu}_{\Gamma,\chi}(\C) }
\newcommand{\Lsdchi}{\mathcal{L}^{2,\nu}_{\Gamma,\chi}(\C) }

\newcommand{\Os}{\mathcal{O}^{\nu}_{\Gamma,\alpha}(\C) }
\newcommand{\Osd}{\mathcal{F}^{2,\nu}_{\Gamma,\alpha}(\C) }
\newcommand{\Lsd}{\mathcal{L}^{2,\nu}_{\Gamma,\alpha}(\C) }

\newcommand{\Em}{\mathcal{E}^{\alpha,\nu}_{m}(\C)} 
\newcommand{\Ell}{\mathcal{E}^{\alpha,\nu}_{0}(\C)}
\newcommand{\Elambda}{\mathcal{E}^{\alpha,\nu}_{\lambda}(\C)}

\newcommand{\Osdg}{\mathcal{F}^{2,\nu}_{\Lat,\chi}(\C) }
\newcommand{\Lsdg}{\mathcal{L}^{2,\nu}_{\Lat,\chi}(\C) }

\newcommand{\Lgamma}{\mathcal{L}^{2}_{\Lat,\alpha}(\R) }

\newcommand{\Kkernel}{K_{\alpha,\nu} }
\newcommand{\transfBarg}{\mathcal{A}_{\nu} }
\newcommand{\Bkernel}{A_{\alpha,\nu} }
\newcommand{\transfG}{\widetilde{\mathcal{A}}_{\nu} }
\newcommand{\Gkernel}{{\widetilde{A}}_{\alpha,\nu} }
\newcommand{\normg}[1]{\left\Vert#1\right\Vert_\Lat}
\newcommand{\normnu}[1]{\left\Vert#1\right\Vert_{\Gamma,\nu}}
\newcommand{\normnua}[1]{\left\Vert#1\right\Vert_{\alpha,\nu}}
 \newcommand{\scalnu}[1]{\left<#1\right>_{\Gamma,\nu}}
 \newcommand{\scalnua}[1]{\left<#1\right>_{\alpha,\nu}}

\newcommand{\epf}{e^{\alpha,\nu}_n }
\newcommand{\eqf}{e^{\alpha,\nu}_m }
\newcommand{\psin}{\psi^{\alpha,\nu}_n }

\newcommand{\mes}{d\lambda}

\newcommand{\pz}{\dfrac{\partial}{\partial z}}

\begin{document}


\title[]{Construction of concrete orthonormal basis for $(L^2,\Gamma,\chi)$-theta functions associated to discrete subgroups of rank one in $(\C,+)$}

 \author{Allal Ghanmi}
  \author{Ahmed Intissar}
 \address{Department of Mathematics, P.O. Box 1014,  Faculty of Sciences, Mohammed V-Agdal University, Rabat, Morocco
}

        \email{ag@fsr.ac.ma} \email{intissar@fsr.ac.ma}


\begin{abstract}
Let $\Gamma$ be a discrete subgroup of rank one of the additive group $(\mathbb{C},+)=(\mathbb{R}^2,+)$,
i.e., $\Gamma=\Z \cdot \omega$ for some $\omega\in \C \setminus\{0\}$, and
 $\chi_\alpha(\gamma) =e^{2i\pi \alpha \gamma}$ a given character on $\Gamma$; $\alpha \in\R$.
By $\Lsd$ we denote the Hilbert space of $(L^2,\Gamma,\chi)$-theta functions,
whose elements are the complex valued functions on the complex plane $\mathbb{C}$ satisfying the functional equation
$$f(z+\gamma) = e^{ 2i\pi \alpha \gamma} e^{\nu H\left(z+\frac \gamma 2,\gamma\right)}f(z); \qquad \nu >0, $$
with $H(z,w):= z \overline{w}$, and such that
$$|| f ||^2_{\Gamma,\nu}:= \int_{\mathbb{C}/\Gamma} |f(z)|^2 e^{-\nu H(z,z)} \mes(z) <+\infty.$$
Our aim is the construction of orthonormal basis for $\Lsd$ as well as for the Hilbert subspace
$\Osd$ of entire functions belonging to $\Lsd$.
For simplicity we deal with the case of  $\Gamma=\Z \cdot 1$ and we show that the functions
$$
\psi_n^{\alpha,\nu} (z) :=  \left({2\nu}/{\pi}\right)^{1/4} e^{-\frac{\pi^2}{\nu}(n+\alpha)^2 } e^{\frac{\nu}2 z^2  + 2i\pi (\alpha + n) z};
 \quad n\in\Z,
$$
 form an orthonormal basis of $\Osd$ and that the corresponding reproducing kernel
$K_{\alpha,\nu}(z,w)$ can be expressed in terms of the generalized theta function $\theta_{\alpha,0}$ of characteristic $[\alpha,0]$.
Furthermore, we show that $\Lsd$ possesses a Hilbertian orthogonal decomposition
  involving the $L^2$-eigenspaces $ \Em $ of the Landau operator
  associated to the eigenvalues $\nu m $ with $ \Ell=\Osd$, and the functions
$$\psi_{m,n}^{\alpha,\nu}(z,\overline{z}) =  \left({2^m m!}\right)^{-1/2}  \psi_n^{\alpha,\nu} (z)
H_m \left(({2\nu})^{1/2} \Im(z) +  ({{2}/{\nu}})^{1/2} \pi (n+\alpha) \right),
$$
for $(m,n)\in \Z^+\times \Z$, constitute a complete orthonormal basis of $\Lsd$, where $H_m$ denotes the $m^{th}$ Hermite polynomial.
\end{abstract}



\maketitle

\section{Introduction, preliminaries and notation}

Classical theta functions are holomorphic functions on $\C$ which are quasi-periodic with respect
to a given lattice $\Lat$ in the complex plane $\C$, i.e., a discrete subgroup of $(\C,+)$ which spans $\C=\R^2$
as a real vector space. This means that for every fixed $\gamma\in \Lat $, there exists a holomorphic map $e_\gamma $ on $\C$ such
that for every $z\in \C$, we have
\begin{equation}\label{IntroFctEq1}
\theta(z+\gamma) =  e_\gamma(z) \theta(z) .
\end{equation}
They can also be interpreted as sections of certain line bundles on complex tori $\C/\Lat $.

The theory of these series is a milestone of nineteenth century mathematics \cite{Houzel1978} and arose from classical problems (length of the ellipse, doubling the arc of the lemniscate).
They attracted the interest of many authors and appeared in Bernoulli's framework (1713) 
as well as in the number-theoretic investigations of Euler (1773) 
and Gauss (1801). 
Special contributions to the theoretical development of these series were made by Jacobi (1827). 
A more detailed theory of elliptic theta series was developed later by Borchardt (1838), 
as well as by  Weil (1948) 
and Weierstrass (1862-1863). 
Classical theta series and related functions are closely connected to the theories of abelian varieties
and quadratic forms, and have been applied to heat conduction problem \cite{Fourier1822},
algorithmic number theory and cryptography and coding theory \cite{Ritzenthaler2004,ShaskaWijesiri2008}.
Moreover, these functions are powerful tools in many areas of quantum field theory
\cite{Cartier1966}.

As basic example, we cite the Jacobi theta function given by the formula
\begin{equation}\label{Jacobi3}
\theta(z|\tau)  := \sum_{n\in \Z} e^{i\pi n^2 \tau + 2i\pi nz} =: \theta_3(z|\tau),
\end{equation}
where $z$ 
 is any complex number and $\tau$ is confined to the upper half-plane $\mathfrak{h}=\{\tau \in \C; \, \Im(\tau)>0\}$.
 In fact, it satisfies the functional equation \eqref{IntroFctEq1} with
 $ e_\gamma(z) := e^{-i\pi l^2\tau - 2i\pi l z}$
 for every $\gamma = l\tau+m $ in the lattice $\Lat = \tau \Z + \Z$, so that
\begin{equation}\label{qpJacobi}
 \theta(z+l\tau + m |\tau) = e^{-i\pi l^2\tau - 2i\pi l z} \theta(z  |\tau) ; \qquad l,m\in \Z.
\end{equation}
To enlarge the category of basic theta series, one considers for any such $\tau$ and any  $\alpha,\beta\in \R$, the series
\begin{equation}\label{RiemannTheta}
 \theta_{\alpha,\beta}(z|\tau)    : =  \sum_{n\in \Z} e^{i\pi (n + \alpha)^2 \tau + 2i\pi (n + \alpha)(z + \beta)}
\end{equation}
which converges absolutely and uniformly for $(z, \tau )$ in compact subsets of $\C\times \mathfrak{h}$.
It can be obtained from $\theta(z|\tau)$ by translating $z$ by $\xi=\alpha \tau + \beta$. In fact, we have
\begin{equation}\label{TransfTheta}
\theta_{\alpha,\beta}(z|\tau) = e^{i\pi \alpha^2\tau + 2i\pi \alpha (z+\beta)} \theta(z+\alpha \tau + \beta|\tau).
\end{equation}
The function $\theta_{\alpha,\beta}(z|\tau)$ was firstly introduced by Riemann \cite{Riemann1857} in connection
with Riemann surfaces and is called the Riemann (modified) theta function (or also first order theta function)
with characteristic [$\alpha,\beta$]. 
The case of integer characteristics leads to the four classical theta series.
An overview of the properties of $ \theta_{\alpha,\beta}$ can be found in
\cite{Krazer70,Mumford83,Mumford84,Mumford91,Baker95,Farkas92}.
The most general form of the Riemann theta function defined above was first considered by Wirtinger \cite{Wirtinger1895}.


Hereafter, we endow $\C$ with its standard hermitian structure defined through $H(z,w)=z\overline{w}$ and denote
 by $\mes(z)=dxdy$; $z=x+iy\in\C$, the Lebesgue measure on $\C$. We emphasize the case of theta functions associated to a discrete subgroup
 $\Gamma$ of $(\C,+)$.
More precisely, we have to consider, among others functional spaces, the vector space $\Oschi$ of
all holomorphic functions on $\C$, $f\in \mathcal{O}(\C)$, displaying the functional equation
\begin{equation}\label{IntrFctEq2}
f(z+\gamma) = \chi(\gamma)  e^{\nu H\left(z+\frac{\gamma}2,\gamma\right) } f(z)
\end{equation}
for every $z\in \C$ and $\gamma\in \Gamma$, where $\nu>0$ and $\chi$ is a given map defined on $\Gamma$
with values in the unit circle $U(1)=\{\lambda \in \C; \, |\lambda|=1\}$.
Also, let $\Lambda(\Gamma)$ be a fundamental domain representing in $\C=\R^2$ the orbital group
$\C/\Gamma$ endowed with the quotient topology and perform the space  $\Osdchi$
of functions $f$ belonging to $\Oschi$ such that
 \begin{equation}\label{IntrNorm}
\int_{\Lambda(\Gamma)} |f(z)|^2 e^{-\nu|z|^2} \mes(z) < +\infty .
\end{equation}
 Note that for $f$ satisfying \eqref{IntrFctEq2}, the function $|f(z)|^2 e^{-\nu|z|^2}$ is $\Gamma$-periodic
 and therefore the quantity in \eqref{IntrNorm} makes sense and is independent of $\Lambda(\Gamma)$. An other functional space to be considered is  $\Lsdchi$ consisting of complex valued functions on $\C$  displaying \eqref{IntrFctEq2} and that are $e^{-\nu|z|^2} \mes$-square
integrable on $\Lambda(\Lat)$ (i.e., such that \eqref{IntrNorm} holds).

Three cases can be envisaged:
 (i) The trivial case of $\Gamma=\{0\}$;
 (ii) The case of $\Gamma=\Z \omega$ for some $\omega\in\C\setminus \{0\}$ and
 (iii) $\Gamma$ is a lattice $\Gamma=\Lat=\Z\omega_1 +\Z\omega_2$ spanned by
$\R$-linearly independent vectors $\omega_1,\omega_2\in \C$. 
This follows, since $\Gamma$ can be viewed as a $\Z$-module of rank $r=0,1,2$.
Thus, for the trivial case $\Gamma=\left\{ 0\right\}$, we have $\Lambda(\Gamma)=\C$ and it is obvious that
\eqref{IntrFctEq2} implies $\chi =1$. Therefore,  $\Osdchi$ reduces to the usual Fock-Bargmann space
consisting of $e^{-\nu|z|^2} \mes$-square integrable entire functions. This space is well studied in the mathematical and physical literatures.
Its relevance to the theory of theta series was revealed by
Cartier \cite{Cartier1966}, Satake \cite{Satake1971} and earlier by Weil \cite{Weil1964}. While,
$\Lsdchi$ reduces to the usual Hilbert space $\mathcal{L}^{2}(\C;e^{-\nu|z|^2} \mes)$, for which we know that
the Hermite polynomials
$$ H_{m,n}(z,\overline{z}) : = (-1)^{m+n} e^{\nu |z|^2} \frac{\partial^{m+n}}{\partial z^n\partial \overline{z}^m} (e^{-\nu |z|^2}); \qquad m,n \in\Z^+,$$
constitute an orthoganal basis \cite{Intissar2}.

For the lattice case; i.e., $\Gamma= \Lat= \Z\omega_1 +\Z\omega_2$, $\Lambda(\Lat)$ is compact, and $\Lsdg$ and $\Osdg$ are  non-trivial Hilbert spaces if and only if
$\chi$ verifies the cocycle condition
$$ \chi(\gamma+\gamma') = e^{i\nu \Im(H(\gamma,\gamma'))} \chi(\gamma) \chi(\gamma') \eqno{(RDQ)}$$
for every $\gamma,\gamma'\in \Lat$.
This amounts that the symplectic form $E(z,w):= (\nu/\pi) \Im(H(z,w))$ takes integer values on $\Lat\times \Lat$.
This is to say that  $(H ; E)$ is a Riemann pair and that the torus $\C/\Lat$ is an abelian variety.
In this case, a systematic study of $\Osdg$ is provided in \cite{GhIn2008}. It is proved there that
its description is closely related to the $L^2$-spectral theory of the Landau Laplacian
\begin{align}
\Delta_\nu =-\frac{\partial^2}{\partial z \partial \overline{z}} + \nu
\overline{z} \frac{\partial}{\partial \overline{z}} .\label{Laplacian}
\end{align}
In fact, $\Osdg$ appears as the null space of $\Delta_\nu$ acting on $\Lsdg$.
Determination of its orthogonal complement in $\Lsdg$  leads to consider the eigenspaces
$
  \mathcal{E}^{\nu,\lambda}_{\Gamma,\chi}(\C) = \{ f\in L^{2,\nu}_{\Gamma,\chi}(\C); \quad \Delta_\nu f = \nu \lambda f  \}.
$ 
Indeed, it is shown in \cite{GhIn2008} that
$
\left ( \Osdg \right)^\perp =  \bigoplus\limits_{m=1}{m=+\infty}  \mathcal{E}^{\nu,m}_{\Gamma,\chi}(\C).
$


In the present paper, we are interested in the non-compact case $\Gamma=\Z \omega$.
Note for instance that in this case the condition $(RDQ)$ holds again and it is a necessary and sufficient condition to the corresponding spaces
$\Lsdchi$ and $\Osdchi$
 be nontrivial. Thence, $\chi$ is a character of $\Z\omega$ and then is completely determined, to wit
$$\chi(\gamma) = e^{2i\pi \alpha m}=: \chi_\alpha(m)$$
for certain fixed real number $\alpha$, where $\gamma = m\omega\in \Z\omega$. Furthermore, the corresponding fundamental domain is the strip $S=[0,1]\times\R$.
To ease notation we consider $\Gamma=\Z$ 
and fix notation for the functional spaces on which we will be working.
For fixed real numbers $\nu>0$ and $\alpha$, we say that a complex valued function $f: \C \longrightarrow \C $ is $(\Z, \chi_\alpha,\nu)$-quasi-periodic, if it satisfies  the functional equation
\begin{equation}\label{FuncEq}
f(z+m) = e^{ 2i\pi \alpha m} e^{\nu (z+\frac{m}2) m} f(z)
\end{equation}
for $z\in \C$ and every $m\in \Z$. Thus, we perform the functional spaces
\begin{equation}\label{FctSpO}
\Os = \left\{ f: \C \longrightarrow \C \mbox{ entire and dispaying }   \eqref{FuncEq}; \, \forall z\in \C, \, \forall  m\in \Z\right\}
\end{equation}
and
\begin{equation}\label{FctSpF}
\Osd = \left\{ f \in \Os; \quad \normnu{f} < +\infty \right\},
\end{equation}
where $\normnu{f}$ is the norm
\begin{equation}\label{norm-def1}
\normnu{f}^2 := \int_{S} |f(z)|^2 e^{-\nu |z|^2 } \mes(z)
\end{equation}
associated to the hermitian inner scalar product
\begin{equation}\label{scalar-def1}
\scalnu{f_1,f_2} := \int_{S} f_1(z)\overline{f_2(z)} e^{-\nu |z|^2 } \mes(z).
\end{equation}
In addition, we have to consider the space
\begin{equation}\label{FctSpL}
\Lsd = \left\{ f: \C \longrightarrow \C \mbox{ satisfying }   \eqref{FuncEq}; \, \forall\forall z\in \C, \, \forall  m\in \Z, \mbox{ and }  \normnu{f} < +\infty  \right\},
\end{equation}
where the notation "$\forall\forall$" is used, {\it \`{a} la Mumford} \cite{Mumford83}, to mean "for almost all".
Then, the space $\Osd$ can be shown to be isometric to the space of holomorphic functions $g(z)=g(x,y)$; $z=x+iy$; $x,y\in \R$,
that are periodic in the $x$-direction and such that
\begin{equation}\label{norm-def2}
\normnua{g}^2 := \int_{x\in [0,1]} \int_{y \in \R} |g(x,y)|^2 e^{-2\nu y^2 - 4\pi \alpha  y} dxdy  < +\infty .
\end{equation}
Similarly, $\Lsd$ is isometric to
\begin{equation}
\left\{g: \R^2 \longrightarrow \C; \, \Z\mbox{-periodic and } \normnua{g}  < +\infty \right\} .
\end{equation}
Furthermore,  orthonormal basis for $\Lsd$ involving Hermite polynomials are constructed explicitly and the concrete description
of $\Lsd$ as well as the orthogonal complement of $\Osd$ in $\Lsd$ are obtained. Indeed, it will be shown that $\Delta_\nu$ in
\eqref{Laplacian}, when acting on $\Lsd$, admits a discrete spectrum consisting of the eigenvalues $m=0,1,2, \cdots$,
which occur with infinite multiplicities, and the following hilbertian spectral decomposition
\begin{align}\label{Split2}
\Lsd =  \bigoplus_{m=0}^{+\infty}  \Em 
  \end{align}
holds, where
$
 \Em = \{ f\in \Lsd ; \quad \Delta_\nu f = \nu m f  \}
$ 
are $L^2$-eigenspaces of $\Delta_\nu$ in $\Lsd$,
  with  $\Ell=\Osd$. Added to this last characterization of $\Osd$ as the null space of $\Delta_\nu$ in $\Lsd$, we have to realize $\Osd$ as the image by the standard Bargmann transform
$$ [ \transfBarg \varphi](z) = c(\nu) \int_{\R} e^{-\frac{\nu}{2}(z^2 -\sqrt2 z x +x^2)} \varphi(x) dx$$
of the Hilbert space $\Lgamma$; $L=\sqrt 2 \Z$, of complex valued functions $\varphi$ on $\R$ satisfying
$$\norm{\varphi}^2:=\int_0^{\sqrt{2}} |\varphi(x)|^2 dx <+\infty$$
and that are $(\sqrt{2}\Z,\chi_\alpha)$-quasi-periodic in the sense that
$$
\varphi(q+\sqrt 2 m)= e^{2i\pi \alpha m} \varphi(q+\sqrt 2 m), \quad \forall\forall  q\in \R, \, \forall m\in \Z.
$$
This constitutes the analogous of the usual fact that $\transfBarg$ maps isometrically the Hilbert space
 $\mathcal{L}^{2}(\R,dx)$ onto $\mathcal{L}^{2}(\C;e^{-\nu|z|^2} \mes)$, which are respectively related to $\Gamma=\{0\} \subset \R$ and $\Gamma=\{0\}\subset \C$.

The layout of the paper is as follows. 
In Section 2, we follows in spirit V. Bargmann \cite{Bargmann61} to give a concrete description of the space $\Osd$. We first give an orthonormal basis
and next obtain some useful results connecting the modified theta series defined in \eqref{RiemannTheta} to $\Osd$.  Section 3 will be devoted to
 establish the link between $\Lgamma$, $L=\sqrt 2 \Z$, and $\Osd$
through the Bargmann transform $\transfBarg$.
In section 4, we study $\Lsd$, discussing the spectral theory of the Landau operator $\Delta_\nu$ when acting on $\Lsd$ and constructing
an orthonormal basis of $\Lsd$ involving Hermite polynomials. 



\section{Description of the Fock-Bargmann type space $\Osd$}

Keep notation as in the introduction and begin with the following.

\begin{proposition}\label{Prop-chara}
Let $f$ be a complex valued function on $\C$. The following assertions are equivalent
\begin{itemize}
 \item [i)]  $f$ belongs to $\Os$.

 \item [ii)] There exists an entire and simply periodic function $g$ (i.e., $g(z+1)=g(z)$ for every $z\in\C$) such that
     \begin{equation}\label{chara}
     f(z):= e^{\frac{\nu}2 z^2 + 2i\pi \alpha z} g(z).
     \end{equation}

 \item [iii)] There exists a unique sequence $(a_n)_{n\in\Z}$ of complex numbers such that
     \begin{equation}\label{expansion}
     f(z):= \sum\limits_{n\in\Z} a_n e^{\frac{\nu}2 z^2 + 2i\pi (\alpha + n) z},
     \end{equation}
     where the series  in \eqref{expansion} is absolutely and uniformly convergent in compact subsets of $\C$.
\end{itemize}
\end{proposition}

\begin{proof}
The proof of $i) \Longrightarrow ii)$ lies on the observation that the function $M^\nu_\alpha(z):= e^{-\frac{\nu}2 z^2 - 2i\alpha \pi z} $ satisfies
$$M^\nu_\alpha(z+m)  
         = e^{- 2i\pi \alpha m} e^{-\nu (z+\frac{m}2) m} M^\nu_\alpha(z)$$
and therefore the function $g(z):= M^\nu_\alpha (z) f(z)$ verifies
$$g(z+m) 
        = e^{- 2i\pi \alpha m} e^{-\nu (z+\frac{m}2) m} M^\nu_\alpha(z) f(z+m)
        = g(z)$$
according to \eqref{FuncEq}. The converse $ii) \Longrightarrow i)$ holds using the fact that
$g$ satisfies $g(z+m) = g(z)$ combined with the same argument on $M^\nu_\alpha(z)$ defined above.
Finally,  $ii) \Longleftrightarrow iii)$ follows upon using the  fact that any entire simply periodic function $g$ admits
a Fourier expansion of the form (see \cite{Bargmann61,Jones-Singerman87}):
$$ g(z) = \sum\limits_{n\in\Z} a_n  e^{ 2i\pi n z}$$
which converges absolutely and uniformly for $z$ in compact subsets of $\C$.
\end{proof}


\begin{remark}\label{Rem-normg}
Straightforward computation, for every $f(z) = e^{\frac\nu 2 z^2 + 2i\pi \alpha z} g(z) \in \Osd$, yields
 \begin{equation}\label{norm-def2g}
\normnu{f}^2 = \int_{x\in [0,1]} \int_{y \in \R} |g(z)|^2 e^{-2\nu y^2 - 4\pi \alpha  y} dxdy =: \normnua{g}^2.
\end{equation}
 Moreover, if we denote by $\scalnua{g_1,g_2}$ the hermitian inner scalar product associated to the norm $\normnua{g}$, then for every $f_j(z) = e^{\frac\nu 2 z^2 + 2i\pi \alpha z} g_j(z) \in \Osd$; $j=1,2$, we have
\begin{equation}\label{scalar-scalrg}
\scalnu{f_1,f_2} = \int_{x\in [0,1]} \int_{y \in \R} g_1(z)\overline{g_2(z)} e^{-2\nu y^2 - 4\pi \alpha  y} dxdy = \scalnua{g_1,g_2}.
\end{equation}
\end{remark}


According to assertion iii) of the above proposition, we perform
\begin{equation}\label{Def-e}
\epf (z):= e^{\frac{\nu}2 z^2 + 2i\pi (\alpha + n) z}; \qquad n\in \Z.
\end{equation}
 Then, we assert the following

\newpage
\begin{theorem}  \label{ThmbasisO}
We have
\begin{itemize}
 \item [i)] The set of functions $\{\epf; \, n\in\Z\}$ constitutes an orthogonal system of $\Osd$ with
     \begin{equation}\label{norm-def3}
     \normnu{\epf}^2 = \left(\dfrac{\pi}{2\nu}\right)^{1/2} e^{\frac{2\pi^2}{\nu}(n+\alpha)^2}.
     \end{equation}

 \item [ii)] 
     A function
     $
     f(z):= 
      \sum\limits_{n\in\Z} a_n \epf (z)
     $ 
     belongs to $\Osd$ if and only if 
     \begin{equation}\label{normf}
     \normnu{f}^2 = \left(\dfrac{\pi}{2\nu}\right)^{1/2} \sum\limits_{n\in\Z} e^{\frac{2\pi^2}{\nu}(n+\alpha)^2}  |a_n|^2  < +\infty .
     \end{equation}
\end{itemize}
\end{theorem}

\begin{proof}
For the proof of i), note first that $\epf(z)= e^{\frac{\nu}2 z^2 + 2i\pi \alpha z} e^{ 2i\pi n z} $ belongs clearly to $\Os$.
Moreover, in view of \eqref{scalar-scalrg}, we get
 \begin{equation}\label{scalarnorm}
\scalnu{\epf ,\eqf} = \left(\int_0^1e^{ 2i\pi(n-m) x} dx\right) \left(\int_{-\infty}^{+\infty}
 e^{-2\nu y^2 - 2\pi (2\alpha +n+m)y} dy\right) .
\end{equation}
Thence, it is obvious that $\scalnu{\epf ,\eqf} = 0$ whenever $n\ne m$, while for $n=m$ \eqref{scalarnorm} reduces to
 \begin{equation*}
\normnu{\epf}^2 =  \int_{-\infty}^{+\infty}  e^{-2\nu y^2 - 4\pi (\alpha+n) y} dy .
\end{equation*}
Hence, we obtain the desired result \eqref{norm-def3} when taking $a=2\nu$ and $b=-4\pi(n+\alpha)$
in the following standard lemma for the Gaussian integrals
\begin{lemma}\label{lem-GaussIntegral}
For every given $a>0$ and $b\in \C$, we have
 \begin{equation} \label{GaussIntegral}
  \int_{-\infty}^{+\infty} e^{-a y^2 + by} dy = \left({\frac{\pi}{a}}\right)^{1/2} e^{b^2/4a} .
  \end{equation}
  \end{lemma}

What needed to prove ii) is the "{\it only if}": According to the previous Proposition \eqref{Prop-chara}, any function $f\in\Osd$ can be written as
     \begin{equation*}
     f(z):=e^{\frac{\nu}2 z^2 + 2i\pi \alpha z} g(z)
     \end{equation*}
     with $\normnu{f}<+\infty$, where $g(z)=g(x,y)$ is a $\Z$-simply periodic function given by
     \begin{equation*}
     g(z) =  \sum\limits_{n\in\Z} a_n e^{ 2i\pi nz} .
     \end{equation*}
     Thus, in view of \eqref{norm-def2g} and upon the use of the Fubini theorem, we obtain
     \begin{align}
     \normnu{f}^2 &= \int_{x\in [0,1]} \int_{y \in \R} |g(x,y)|^2 e^{-2\nu y^2 - 4\pi \alpha  y} dxdy\nonumber \\
                        &=  \int_{-\infty}^{+\infty}\left(\int_0^1 |g(x,y)|^2 dx\right) e^{-2\nu y^2 - 4\pi \alpha  y} dy . \label{Fubini}
     \end{align}
     Therefore, for every fixed $y\in\R$, we have
     $$ \int_0^1 |g(x,y)|^2  dx <+\infty.$$
    This means that the periodic function (in the $x$-direction)
     $$ x \longmapsto g(x,y) = \sum\limits_{n\in\Z} \left( a_n e^{- 2\pi ny}\right) e^{ 2i\pi nx} $$
     belongs to the Hilbert space $L^2([0,1];dx)$, where the series converges uniformly on 
     $[0,1]$ for every fixed $y\in\R$.
     Thus, we may apply the Parseval identity to get
      $$ \int_0^1 |g(x,y)|^2 dx =  \sum\limits_{n\in\Z} |a_n|^2  e^{- 4\pi ny}.$$
     Substituting this in \eqref{Fubini} leads to
     \begin{align*}
     \normnu{f}^2  &=  \sum\limits_{n\in\Z} |a_n|^2  \int_{-\infty}^{+\infty}  e^{-2\nu y^2 - 4\pi (\alpha+n)  y} dy\\
                   &=  \sum\limits_{n\in\Z} |a_n|^2  \normnu{\epf}^2 \\
                   &= \left(\dfrac{\pi}{2\nu}\right)^{1/2} \sum\limits_{n\in\Z} e^{\frac{2\pi^2}{\nu}(n+\alpha)^2} |a_n|^2.
     \end{align*}
This proves ii).
\end{proof}

The following result shows that $\Osd$ is a reproducing kernel Hilbert space, whose reproducing kernel can expressed explicitly in terms of the modified theta function $\theta_{\alpha,0}$. Namely, we have

\begin{theorem}
\begin{itemize}
\item [i)] Let $f\in \Osd$. Then, for every $z\in\C$, we have the following estimation
     \begin{equation}\label{estimation}
     |f(z)| \leq  \normnu{f}  \left( \left(\dfrac{2\nu}{\pi}\right)^{1/2} e^{\frac{\nu}2 (z^2+\overline{z}^2)} \theta_{\alpha,0} \left(z-\overline{z} \bigg | \frac{2i\pi}{\nu} \right)\right)^{1/2}.
     \end{equation}
     In particular for every compact set $K$ of $\C$ there exists a constant $C_K\geq 0$ such that
     \begin{equation}\label{estimationK}
     |f(z)| \leq  C_K \normnu{f}; \qquad z\in K.
     \end{equation}

\item [ii)] $\Osd$ is a Hilbert space with the orthonormal basis
\begin{equation}\label{orthonorm-basis} 
\psin (z) :=  \left(\dfrac{2\nu}{\pi}\right)^{1/4}  e^{\frac{\nu}2 z^2} e^{-\frac{\pi^2}{\nu}(n+\alpha)^2 + 2i\pi (\alpha + n) z}; \quad n\in\Z.
\end{equation}

\item [iii)] $\Osd$ admits a reproducing kernel function $\Kkernel(z,w)$, given explicitly in terms of the Reimann theta function $\theta_{\alpha,0}$, as
  \begin{equation}\label{Kernel}
 \Kkernel(z,w) = \left(\dfrac{2\nu}{\pi}\right)^{1/2} e^{\frac{\nu}{2}(z^2+ {\overline{w}}^2)}
 \theta_{\alpha,0} \left(z-\overline{w} \bigg |  \dfrac{2i\pi}{\nu}\right).
 \end{equation}
 \end{itemize}
 \end{theorem}

 \begin{proof}
The estimation \eqref{estimation} in i) can be handled using the Cauchy-Schwarz inequality. Indeed, for $f= \sum\limits_{n\in\Z} a_n \epf (z) \in \Osd$, we have
\begin{align*} |f(z)| 
                     &\leq  \left(\sum\limits_{n\in\Z} |a_n|^2 \normnu{\epf}^2 \right)^{1/2}
                      \left(\sum\limits_{n\in\Z} |\psin (z)|^2   \right)^{1/2}\\
                     & \leq  \normnu{f}
                     \left( \left(\dfrac{2\nu}{\pi}\right)^{1/2} |e^{\frac{\nu}2 z^2 }|^2 \sum\limits_{n\in\Z}   e^{-\frac{2\pi^2}{\nu}(n+\alpha)^2 + 2i\pi (\alpha + n) (z-\overline{z}) }  \right)^{1/2}.
\end{align*}
The sum in the right hand side is bounded on compact subsets of $\C$ and prepared to be written in terms of the modified theta series defined by
\eqref{RiemannTheta}.
In fact, for $\beta=0$ and $\tau=2i\pi /\nu$, we obtain
$$
|f(z)| \leq   \normnu{f} \left( \left(\dfrac{2\nu}{\pi}\right)^{1/2} e^{\frac{\nu}2 (z^2+\overline{z}^2)}
\theta_{\alpha,0} \left(z-\overline{z} \bigg | \frac{2i\pi}{\nu} \right)\right)^{1/2} .
$$
This ends the proof of i).

 What needed to prove ii) is completeness. For this let $(f_p)_{p}$ be a Cauchy sequence in $\Osd$.
  Then from i), we know that for any compact set $K\subset \C$, there is $C_K$ such that
$$|f_p(z) - f_q(z)| \leq C_K \normnu{f_p-f_q}.$$
Whence, the sequence $(f_p)_{p}$ of entire functions satisfying the functional equation
\eqref{FuncEq} is uniformly Cauchy on compact subsets, and converges pointwise and uniformly
to a function $f$ defined on the whole $\C$, which is entire and satisfies
  \begin{equation*}
f(z+m) = e^{ 2i\pi \alpha m} e^{\nu (z+\frac{m}2) m} f(z); \quad \forall z\in\C, \, \forall m\in\Z.
\end{equation*}
 We need only to prove that $\normnu{f} < +\infty$. For this, note that since
$$
\normnu{f_p}^2= \int_{S} |f_p(z)|^2 e^{-\nu|z|^2} \mes(z) < +\infty,
$$
then $f_p|_{_{S}}$ belongs to the Hilbert space $\mathcal{H}_{S}:= L^2(S;e^{-\nu|z|^2} \mes)$.
Thence, $(f_p)_{p}$ is also a Cauchy sequence in $\mathcal{H}_{S}$. Therefore,
$(f_p)_{p}$ converges to a function $\phi_S \in \mathcal{H}_{S}$ in the norm of $\mathcal{H}_{S}$.
 Furthermore, there exists a subsequence $(f_{p_k})_k$ of $(f_p)_p$  converging to $\phi_S$ pointwise almost everywhere on $S$.
 Thus, we have $ f|_{_{S}} =  \phi_S \in \mathcal{H}_{S}$  almost everywhere on $S$, and therefore,
 $$
 \int_{S} |f(z)|^2 e^{-\nu|z|^2} \mes(z) = \int_{S} |\phi_S(z)|^2 e^{-\nu|z|^2} \mes(z) < +\infty .
 $$
 To complete the proof of ii), we normalize the orthogonal system $\epf$ (see \eqref{Def-e} and \eqref{norm-def3}), so as
 \begin{equation*}\label{orthonorm-basis1} 
\psin (z) :=  \left(\dfrac{2\nu}{\pi}\right)^{1/4}  e^{\frac{\nu}2 z^2} e^{-\frac{\pi^2}{\nu}(n+\alpha)^2 + 2i\pi (\alpha + n) z}
\end{equation*}
form an orthonormal system, which turns out to be a basis of the Hilbert space $\Osd$ according to iii) of Proposition \ref{Prop-chara}.

 The proof of iii) follows using again \eqref{estimationK}, which means that the evaluation map $f \mapsto f(z)$ is continuous,
 and therefore $\Osd$ is a reproducing kernel Hilbert space. Its reproducing kernel function is given by
$ \Kkernel(z,w) = \sum\limits_{n\in\Z} \psin (z) {\overline{\psin (w)}}$.
Moreover, we have
\begin{align*}  \Kkernel(z,w)  &=  \left(\dfrac{2\nu}{\pi} \right)^{1/2} e^{\frac{\nu}2 (z^2+{\overline{w}}^2)}
             \sum_{n\in\Z}   e^{-\frac{2\pi^2}{\nu}(n+\alpha)^2 + 2i\pi (\alpha + n) (z-\overline{w})}   \\
     &= \left(\dfrac{2\nu}{\pi}\right)^{1/2}  e^{\frac{\nu}{2}(z^2+ {\overline{w}}^2)} \theta_{\alpha,0}
     \left(z-\overline{w} \bigg |  \dfrac{2i\pi}{\nu}\right).
\end{align*}
\end{proof}

\begin{remark}
The multiplication operator $[M_{\nu,\alpha} (f)](z) := e^{\frac{\nu}2 z^2 + 2i\pi  \alpha   z} f(z)$
defines an isometric mapping from the Hilbert space $\Osd$ onto its companion  the Hilbert space of all entire simply periodic functions that are square integrable with respect to the norm $\normnua{g}$ in \eqref{norm-def2g}.
\end{remark}

We conclude this section with the following

\begin{proposition}\label{prop-thetaO} Let $\alpha,\beta \in \R$ and $\tau \in \C$ such that $\Im(\tau)>0$. Then, the holomorphic function
 \begin{equation}\label{holom-theta}
 f^{\nu,\tau}_{\alpha,\beta}(z):=  e^{\frac \nu 2 z^2} \theta_{\alpha,\beta}(z|\tau)
 \end{equation}
 belongs to $\Osd$ if and only if $\Im(\tau)> \pi/\nu$.
\end{proposition}

 \begin{proof}
 Using the fact that
 $ \theta_{\alpha,\beta}(z+m|\tau) = e^{2i\pi \alpha m} \theta_{\alpha,\beta}(z|\tau) $,
 we get easily by direct computation that $f^{\nu,\tau}_{\alpha,\beta}(z)$ satisfies the quasi-periodic property
 $$ f^{\nu,\tau}_{\alpha,\beta}(z+m) = e^{2i\pi \alpha m} e^{\frac \nu 2 z^2 + \nu\left(z+\frac m2\right)m } f^{\nu,\tau}_{\alpha,\beta}(z)$$
 for every $m\in\Z$. Note also that $f^{\nu,\tau}_{\alpha,\beta}(z)$ can be rewritten in the following form
  \begin{equation}
 f^{\nu,\tau}_{\alpha,\beta}(z):= e^{2i\pi\alpha\beta} e^{\frac \nu 2 z^2 + 2i\pi \alpha z } g_{\alpha,\beta}^{\nu,\tau}(z),
 \end{equation}
 where $ g_{\alpha,\beta}^{\nu,\tau}$ is the $\Z$-simply periodic holomorphic function given by
   \begin{equation}
  g_{\alpha,\beta}^{\nu,\tau}(z) =  \sum_{n\in\Z} e^{i\pi(n+\alpha)^2\tau + 2i\pi \beta n } e^{2i\pi n z }.
 \end{equation}
 According to \eqref{norm-def2g} in Remark \ref{Rem-normg}, we obtain
  \begin{align*}
\normnu{ f^{\nu,\tau}_{\alpha,\beta}}^2
&= \int_{x\in [0,1]} \int_{y \in \R} |g_{\alpha,\beta}^{\nu,\tau}(x,y)|^2 e^{-2\nu y^2 - 4\pi \alpha  y} dxdy \\
  &= \sum_{n\in\Z} |e^{i\pi(n+\alpha)^2\tau + 2i\pi \beta n }|^2   \int_{x\in [0,1]} \int_{y \in \R} |e^{2i\pi n z }|^2 e^{-2\nu y^2 - 4\pi \alpha  y} dxdy \\
  &= \sum_{n\in\Z} e^{-2\pi(n+\alpha)^2\Im(\tau)}  \int_{-\infty}^{+\infty}  e^{-2\nu y^2 - 4\pi(n+ \alpha)  y} dy,
\end{align*}
which, in view of \eqref{GaussIntegral}, gives rise to
   \begin{align*}
\normnu{ f^{\nu,\tau}_{\alpha,\beta}}^2
= \left(\frac{\pi}{2\nu}\right)^{1/2} \sum_{n\in\Z}  e^{-2\pi(n+\alpha)^2\left(\Im(\tau) - \frac{\pi}{\nu}\right)}  .
\end{align*}
Now using the well established fact that the convergence of the series $\sum\limits_{n\in\Z} e^{-a(n+\alpha)^2}$ is equivalent to $a>0$, we deduce that
$ f^{\nu,\tau}_{\alpha,\beta} \in \Osd$ if and only if $\Im(\tau) - ({\pi}/{\nu}) >0$. This completes the proof.
\end{proof}

\begin{remark}
Making appeal to the previous result 
and the reproducing property satisfied by the reproducing kernel $ \Kkernel(z,w)$, i.e.,
   \begin{equation}\label{Kernel-integral}
 f(z) =  \int_{w\in S}   \Kkernel(z,w) f(w) e^{-\nu |w|^2} \mes(w) ,
 \end{equation}
 for every $f\in \Osd$, one can assert that for every $z,\tau\in\C$ such that $\Im(\tau) >\pi/\nu$ and reals numbers $\alpha,\beta$, we have
\begin{equation*}
  \int_{w\in S} \theta_{\alpha,0} \left(z-\overline{w} \bigg |  \dfrac{2i\pi}{\nu}\right)  \theta_{\alpha,\beta} (w|\tau)
 e^{\frac{\nu}{2}{\overline{w}}^2-\nu|w|^2} \mes(w) = \left(\dfrac{\pi}{2\nu}\right)^{1/2} e^{-\nu z^2} \theta_{\alpha,\beta} (z|\tau) .
\end{equation*}
\end{remark}


\section{Connection of $\Osd$ to $\Lgamma$ through the Bargmann transform}

Fix $\alpha\in\R$ and set
$$ \chi_\alpha(m) : = e^{2i\pi \alpha m}; \quad m\in \Z.$$
Associated to the lattice $L:=\sqrt{2} \Z$ of $\R$, let denote by $\Lgamma$ 
the space of complex valued measurable functions on $\R$ satisfying the quasi-periodicity
     \begin{equation}\label{qp2}
        \varphi(q+\sqrt{2} m) = \chi_\alpha(m) \varphi(q);  \qquad \forall \forall  q\in\R, \forall m\in\Z,
     \end{equation}
and such that
        \begin{equation}\label{normVarphi}
        \normg{\varphi}^2 :=\int_0^{\sqrt{2}} |\varphi(q)|^2 dq <+\infty.
        \end{equation}
Then, it is easy to check the following proposition (whose the proof is omitted).

\begin{proposition}\label{prop-qpVarphi}
\begin{enumerate}
    \item[i)] A measurable function $\varphi:\R \longrightarrow \C$ belongs to $\Lgamma$ if and only if it can be written as
        \begin{equation}\label{fctVarphi}
       \varphi(q) = e^{\sqrt{2} i \pi \alpha q} g(q)
         \end{equation}
        for some $\sqrt{2}\Z$-periodic function $g$ (i.e., $g(q+\sqrt{2}m) = g(q)$; $\forall\forall q\in \R$ and $\forall m\in\Z$) such that
        $ \normg{\varphi}<+\infty$.
    \item[ii)]  $\Lgamma$ is a Hilbert space with the orthonormal basis
        \begin{equation}\label{fctVarphin}
       \varphi_n^{\alpha}(q):=2^{-1/4} e^{\sqrt{2}i\pi (n+\alpha)q} ; \quad q\in\R, n\in\Z.
        \end{equation}
    \item[iii)]
%
%
  Every $\varphi \in \Lgamma$ can be expanded as
     \begin{equation}\label{expVarphi}
      \varphi (q):= \sum\limits_{n\in\Z} a_n \varphi_n^{\alpha}(q),
      \end{equation}
 where $(a_n)_{n\in\Z}\subset \C$ is such that $\sum\limits_{n\in\Z} |a_n|^2 < +\infty$.
  \end{enumerate}
\end{proposition}


To state the main result of this section, we need to introduce the following

\begin{definition}\label{DefBargmTransf}
The Bargmann transform $\transfBarg$ of given $\varphi:\R \longrightarrow \C$
is defined as a function on $\C$ by
        \begin{align}\label{TransfBarg}
        [\transfBarg\varphi](z) = \left(\frac{\nu}{\pi}\right)^{3/4}
         \int_{-\infty}^{+\infty} \varphi(q) e^{-\frac{\nu}{2}(z^2+q^2) +\nu\sqrt{2}z q} dq
        \end{align}
or equivalently
        \begin{align}\label{TransfBargEqv}
        [\transfBarg\varphi](z) = \left(\frac{\nu}{\pi}\right)^{3/4} e^{\frac{\nu}{2}z^2}
         \int_{-\infty}^{+\infty} \varphi(q) e^{-\frac{\nu}{2}(q - \sqrt{2}z)^2} dq
        \end{align}
        provided that the integrals exist. This is the case if $\varphi$ is assumed to be a bounded function on $\R$ satisfying \eqref{qp2}.
\end{definition}

Thus the quantity $[\transfBarg \varphi_n^{\alpha}]$ is well defined
 and its explicit expression is given by the following

\begin{proposition}
We have
  \begin{equation}\label{TransfBargn}
      [\transfBarg \varphi_n^{\alpha}](z) = \left(\frac{2\nu}{\pi}\right)^{1/4}  e^{-\frac{\pi^2}{\nu} (n+\alpha)^2}\epf(z)
       =   \psin (z) .
      \end{equation}
\end{proposition}

\begin{proof} Insertion of \eqref{fctVarphin} in \eqref{TransfBargEqv} infers
        \begin{align*}
        [\transfBarg\varphi_n^{\alpha}](z)
         &=  2^{-1/4}\left(\frac{\nu}{\pi}\right)^{3/4}e^{\frac{\nu}{2}z^2}\int_{-\infty}^{+\infty}
         e^{-\frac{\nu}{2}(q - \sqrt{2}z)^2 + \sqrt{2}i\pi (n+\alpha)q} dq\\
        &=  2^{-1/4}\left(\frac{\nu}{\pi}\right)^{3/4}e^{\frac{\nu}{2}z^2}e^{2i\pi (n+\alpha)z}\int_{-\infty}^{+\infty}
        e^{-\frac{\nu}{2}(q - \sqrt{2}z)^2 + \sqrt{2}i\pi (n+\alpha)(q - \sqrt{2}z)} dq.
     \end{align*}
Now, in view of \eqref{GaussIntegral} with $a=\nu/2$, $b= \sqrt{2}i\pi(n+\alpha) $ and $y = q - \sqrt{2}z$, it follows
        \begin{align*}
        [\transfBarg\varphi_n^{\alpha}](z) & =  \left(\frac{2\nu}{\pi}\right)^{1/4}  e^{-\frac{\pi^2}{\nu} (n+\alpha)^2}\epf(z) =     \psin (z).
        \end{align*}
\end{proof}

\begin{remark}
The Bargmann transform $\transfBarg$ maps the orthonormal basis $\varphi_n^{\alpha}$
 of the Hilbert space $\Lgamma$ onto the orthonormal $ \psin $ of the Hilbert space $\Osd$. This means that
  $\transfBarg$ is an isometric embedding from $\Lgamma$ into $\Osd$.
  \end{remark}

Some basic properties of the Bargmann transform $\transfBarg: \Lgamma \longrightarrow \Osd$ are discussed in the following

\begin{theorem} Let  $\varphi$ be a bounded function on $\R$ satisfying the ${\sqrt 2}\Z$-quasi-periodicity \eqref{qp2}. Then, we have
 \begin{itemize}
 \item[i)] The function $\psi^{\alpha,\nu} (z):=[\transfBarg \varphi](z)$ is entire and satisfies the functional equation
   \begin{equation}\label{FctEqBargm}
      \psi^{\alpha,\nu} (z+m)=\chi_\alpha(m) e^{\nu(z+\frac{m}{2})m} \psi^{\alpha,\nu} (z); \qquad \forall z\in\C, \, \forall m\in\Z.
      \end{equation}
 \item[ii)]  Let $\theta_{3}(\xi|\tau)=\theta(\xi|\tau)$ be the Jacobi theta series defined through \eqref{Jacobi3} and set
      \begin{equation}\label{KernelBargm}
      \Bkernel(z;q) :=  \left(\frac{\nu}{\pi}\right)^{3/4} e^{\frac{\nu}{2}z^2 - \nu\left(\frac{q}{\sqrt 2}-z\right)^2} \theta_{3}\left( \frac{i\nu}{\pi}\left(\frac{q}{\sqrt 2}-z\right) +\alpha \bigg| \frac{i\nu}{\pi}\right)
      \end{equation}
      for every $(z;q)\in \C\times \R$. Then, we have
      \begin{equation}\label{TransfBargm2}
      [\transfBarg\varphi](z)=  \int_0^{\sqrt{2}} \Bkernel(z;q) \varphi(q) dq; \quad z\in\C.
      \end{equation}
 \item[iii)] The kernel function $\Bkernel$ in \eqref{KernelBargm} displays
      \begin{equation}\label{TransfBargm2}
      \Bkernel(z+m;q+\sqrt{2}m')=\chi_\alpha(m) e^{\nu(z+\frac{m}{2})m} \Bkernel(z;q) \overline{\chi_\alpha(m') }
      \end{equation}
      for every $z\in\C$, $q\in\R$ and $m,m'\in\Z$.
 \end{itemize}
\end{theorem}

\begin{proof}
Since $(z,q) \mapsto \varphi(q) e^{-\frac{\nu}{2}(z^2+q^2) +\nu\sqrt{2}z q}$ is entire in $z$ and the integral involving this function is uniformly convergent
{on compact subsets of $\C\times \R$} (for $\varphi$ being bounded on $\R$), we conclude that $\psi^{\alpha,\nu} (z)= \transfBarg\varphi(z)$ is entire. Now, by writing
$\psi^{\alpha,\nu} (z+m):=[\transfBarg \varphi](z+m)$ as integral, using \eqref{TransfBargEqv}, we get
        \begin{align*}
        \psi^{\alpha,\nu} (z+m)
        &= \left(\frac{\nu}{\pi}\right)^{3/4} e^{\frac{\nu}{2}z^2}e^{\nu\left(z+\frac{m}{2}\right)m}  \int_{-\infty}^{+\infty} \varphi(q) e^{-\frac{\nu}{2}(q - \sqrt{2}(z+m))^2} dq.
        \end{align*}
Replacing $q$ by $q+\sqrt{2}m$ in the above integral, yields
         \begin{align*}
        \psi^{\alpha,\nu} (z+m)
        &= \left(\frac{\nu}{\pi}\right)^{3/4} e^{\frac{\nu}{2}z^2}e^{\nu\left(z+\frac{m}{2}\right)m}  \int_{-\infty}^{+\infty} \varphi(q+\sqrt{2}m) e^{-\frac{\nu}{2}(q - \sqrt{2}z)^2} dq.
        \end{align*}
Therefore, since $\varphi(q+\sqrt{2}m)=\chi_\alpha(m)\varphi(q)$, we obtain
         \begin{align*}
        \psi^{\alpha,\nu} (z+m)
        &= \chi_\alpha(m)e^{\nu\left(z+\frac{m}{2}\right)m} \left(\frac{\nu}{\pi}\right)^{3/4} e^{\frac{\nu}{2}z^2}  \int_{-\infty}^{+\infty} \varphi(q) e^{-\frac{\nu}{2}(q - \sqrt{2}z)^2} dq\\
         &= \chi_\alpha(m)e^{\nu\left(z+\frac{m}{2}\right)m}\psi^{\alpha,\nu} (z).
        \end{align*}

To prove ii), we make use of the facts $\R=\bigcup\limits_{m\in\Z} (\sqrt{2} m + [0,\sqrt{2}])$ and $\varphi(q+\sqrt{2}m)=\chi_\alpha(m)\varphi(q)$, to  rewrite $[\transfBarg\varphi](z)$ in \eqref{TransfBargEqv} as
        \begin{align*}
        \psi^{\alpha,\nu} (z)
        &= \left(\frac{\nu}{\pi}\right)^{3/4} e^{\frac{\nu}{2}z^2}  \sum_{m\in\Z}\int_{0}^{\sqrt{2}} \varphi(q+\sqrt{2}m) e^{-\frac{\nu}{2}(q+\sqrt{2}m - \sqrt{2}z)^2} dq \\
        &=   \int_{0}^{\sqrt{2}}\varphi(q) \left(\left(\frac{\nu}{\pi}\right)^{3/4} e^{\frac{\nu}{2}z^2}\sum_{m\in\Z} \chi_\alpha(m) e^{-\frac{\nu}{2}((q - \sqrt{2}z)+\sqrt{2}m)^2}\right) dq .
        \end{align*}
What needed is to give a closed expression of the sum in the right hand side in terms of some special functions. In fact, we have
      \begin{align*}
      \sum_{m\in\Z} \chi_\alpha(m) e^{-\frac{\nu}{2}((q - \sqrt{2}z)+\sqrt{2}m)^2}
       & =
       e^{-\frac{\nu}{2}(q - \sqrt{2}z)^2}\sum_{m\in\Z} \chi_\alpha(m) e^{-\nu m^2 - 2\nu\left(\frac{q}{\sqrt{2}}- z\right)m }\\
       & = 
       e^{-\frac{\nu}{2}(q - \sqrt{2}z)^2}\sum_{m\in\Z}  e^{i\pi m^2\left(\frac{i\nu}{\pi}\right) + 2i\pi \left(\frac{i\nu}{\pi}\left(\frac{q}{\sqrt{2}} - z\right)+\alpha\right) m }.
       \end{align*}
Hence in view of \eqref{Jacobi3}, we deduce
       \begin{equation*}
      \sum_{m\in\Z} \chi_\alpha(m) e^{-\frac{\nu}{2}((q - \sqrt{2}z)+\sqrt{2}m)^2}
      = 
       e^{-\nu\left(\frac{q}{\sqrt 2}-z\right)^2} \theta_{3}\left( \frac{i\nu}{\pi}\left(\frac{q}{\sqrt2}-z\right) +\alpha \bigg| \frac{i\nu}{\pi}\right) .
      \end{equation*}
  This completes the proof of ii).

 Finally, the assertion iii) follows by combining
      $$
      \Bkernel(z+m;y)=\chi_\alpha(m) e^{\nu(z+\frac{m}{2})m} \Bkernel(z;y)
      $$
  which is implicitly contained in the proof of i), and the fact
    \begin{align}\label{InvJacobi2}
      \Bkernel(z;q+\sqrt{2}m') =\Bkernel(z;q) \overline{\chi_\alpha(m')}.
      \end{align}
  To prove \eqref{InvJacobi2}, we write first $\Bkernel(z;q+\sqrt{2}m')$ (in \eqref{KernelBargm}) in terms of
   $\tau =\frac{i\nu}{\pi} $ and $Z= \left(\frac q{\sqrt{2}}- z\right)$ as
     \begin{align*}
       \Bkernel(z;q+\sqrt{2}m')
       &=\left(\frac{\nu}{\pi}\right)^{3/4} e^{\frac{\nu}{2}z^2 -\nu Z^2} e^{i\pi m'^2\tau + 2i\pi \tau Z m' } \theta_{3}\left(\tau Z + \alpha +\tau m' \bigg| \tau\right)
      \end{align*}
 and next use the transformation \eqref{qpJacobi}, to wit
  \begin{align*}
   \theta_{3}\left(\tau Z + \alpha +\tau m' \bigg| \tau\right)
   &= e^{-i\pi m'^2\tau - 2i\pi (\tau Z+\alpha) m' } \theta_{3}\left(\tau Z + \alpha  \bigg| \tau\right)\\
   &= e^{-i\pi m'^2\tau - 2i\pi \tau Z m' } \theta_{3}\left(\tau Z + \alpha  \bigg| \tau\right) \overline{\chi_\alpha(m')} .
    \end{align*}
    This ends the proof.
\end{proof}


An other way to construct integral transform between Hilbert spaces is to consider the kernel function defined as the
bilateral generating function of their orthonormal basis. In our case of the Hilbert spaces
$\Lgamma$ and $\Osd$, we define
        \begin{align}\label{BilateralGen}
        \Gkernel(z;q) & := \sum_{n\in\Z}  \psi^{\alpha,\nu}_n(z) \overline{\varphi_n^{\alpha}(q)}
        \end{align}
and its corresponding integral transform
        \begin{align}\label{TransfG}
        [\transfG \varphi](z) =  \int_{0}^{\sqrt{2}} \Gkernel(z;q) \varphi(q)  dq; \quad z\in \C,
        \end{align}
for given $\varphi:\R \longrightarrow \C$, provided that the integral exists.
Then it will be shown that $\transfG$ is in fact the Bargmann transform $\transfBarg$ on $\Lgamma$ (see Definition \ref{DefBargmTransf}).
 Namely, we have the following

\begin{theorem}
 \begin{itemize}
        \item[i)] For every $n\in \Z$, we have
      \begin{equation}\label{TransfGn}
        [\transfG \varphi_n^{\alpha}](z) =     \psin (z)
      \end{equation}
    Furthermore, $\transfG=\transfBarg$ and $\Gkernel(z;q)=\Bkernel(z;q)$.

 \item[ii)] The kernel function $\Gkernel(z;q)$ is given explicitly by
    \begin{align}\label{BilateralGenExp}
        \Gkernel(z;q)   &= \left(\frac{\nu}{\pi}\right)^{1/4} e^{\frac{\nu}2 z^2} \theta_{\alpha,0}\left( z-\frac{q}{\sqrt 2} \bigg |\frac{i\pi}{\nu}\right).
        \end{align}
 \end{itemize}
\end{theorem}

\begin{proof}
 By definition of $\Gkernel$, we get
      \begin{align*}
        [\transfG \varphi_n^{\alpha}](z)
        &= \int_{0}^{\sqrt{2}} \left(\sum_{k\in\Z}  \psi^{\alpha,\nu}_k(z) \overline{\varphi_k^{\alpha}(q)}\right) \varphi_n^{\alpha}(q)  dq \\
        &= \sum_{k\in\Z}  \psi^{\alpha,\nu}_k(z) \left(\int_{0}^{\sqrt{2}} \varphi_n^{\alpha}(q)\overline{\varphi_k^{\alpha}(q)} dq\right) .
      \end{align*}
      Since $\varphi_n^{\alpha}$ is an orthonormal basis of $\Lgamma$, it follows $[\transfG \varphi_n^{\alpha}](z) =\psin$ and therefore
       $\transfG=\transfBarg$, according to \eqref{TransfBargn}. This leads in particular to
       $$ \int_{0}^{\sqrt{2}} [\Gkernel-\Bkernel](z;q)  \varphi_n^{\alpha} (q) dq =  0; \quad n\in\Z.$$
        Thus, since $\{\varphi_n^{\alpha}; n\in\Z\}$ is complete orthonormal system in $\Lgamma$, we deduce that
       the function $ q \longmapsto [\Gkernel-\Bkernel](z;q) $ is identically zero for every fixed $z\in \C$. Whence $\Gkernel(z;q) =\Bkernel(z;q) $ on $\C\times \R$.
       This ends the proof of assertion i).

Finally, ii) can be checked using the explicit expressions of $\psin$ \eqref{orthonorm-basis} and of $\varphi_n^{\alpha}$ \eqref{fctVarphin}. Thus, from the definition we can rewrite $\Gkernel(z;q)$ as
    \begin{align*}
        \Gkernel(z;q) 
                      & =\sum_{n\in\Z}   \left(\frac{2\nu}{\pi}\right)^{1/4}  e^{\frac{\nu}2 z^2} e^{-\frac{\pi^2}{\nu} (n+\alpha)^2 + 2i\pi (\alpha + n) z}
                      \left(2^{-1/4} e^{-\sqrt{2}i\pi (n+\alpha)q}\right)\\
                      & = \left(\frac{\nu}{\pi}\right)^{1/4}  e^{\frac{\nu}2 z^2 } \sum_{n\in\Z}    e^{i\pi  (n+\alpha)^2\left(\frac{i\pi}{\nu}\right) +                     2i\pi  (n+\alpha) ( z-\frac q{\sqrt{2}})}\\
                      &= \left(\frac{\nu}{\pi}\right)^{1/4} e^{\frac{\nu}2 z^2}  \theta_{\alpha,0}\left( z-\frac{q}{\sqrt 2} \bigg |\frac{i\pi}{\nu}\right).
        \end{align*}
\end{proof}

\begin{remark}
A direct proof of $\Gkernel(z;q)=\Bkernel(z;q)$ can be given starting from the expression of the kernel function $\Gkernel(z;q)$, see \eqref{BilateralGenExp}, and next applying the transformation
 \eqref{TransfTheta} for the Reimann theta series $\theta_{\alpha,\beta}$ combined with the inversion formula
 $$ \theta_3(\xi|\tau) = \left(\frac{i}{\tau} \right)^{1/2} e^{-i\pi\xi^2/\tau} \theta_3\left(\frac{\xi}{\tau}\bigg |-\frac{1}{\tau}\right).$$
 In fact $\Gkernel(z;q)=\Bkernel(z;q)$ and the cited inversion formula for $\theta_3$ are equivalent.
\end{remark}

\begin{remark}
Since $\transfG$ is an unitary operator, then $[\transfG]^{-1}$ coincides with the adjoint of $\transfG: \Osd \longrightarrow \Lgamma$.
We have $$ [\transfG]^{-1} (\psi)(q) : = \int_{S} \overline{\Gkernel(z;q)} \psi(z) e^{-\nu|z|^2}\mes(z) = \scalnu{\psi, \Gkernel(\cdot;q)}$$
which is well defined on a dense subspace of $\Osd$. This then is a new inversion formula for the Bargmann transform $\transfBarg$.
\end{remark}


\section{Description of the space $\Lsd$}
The aim of this section is twofold. First, we give an orthonormal basis of the space $\Lsd$, which can be shown to be a Hilbert space,
whose elements are the complex valued functions on the complex plane $\mathbb{C}$ satisfying the functional equation
\begin{align} \label{FuncEq00}
f(z+m) = e^{ 2i\pi \alpha m} e^{\nu \left(z+\frac m 2\right)m}f(z); \qquad \forall\forall z \in \C, \, \forall m\in \Z,
\end{align}
 and such that
$$|| f ||^2_{\Gamma,\nu}:= \int_{S} |f(z)|^2 e^{-\nu |z|^2} dm(z) <+\infty.$$
 Second, we determine the orthogonal complement of $\Osd$ in $\Lsd$.

\begin{theorem} \quad
The family of functions $\psi_{m,n}^{\alpha,\nu}(z,\overline{z})$; $(m,n)\in \Z^+\times \Z$,
given in terms of the $m^{th}$ Hermite polynomial $H_m$ as
\begin{align}\label{CloHer}
\psi_{m,n}^{\alpha,\nu}(z,\overline{z}) &=  C_{m,n}^{\alpha,\nu}     e^{\frac{\nu}2 z^2 + 2i\pi (\alpha + n) z}
                              H_m \left(\sqrt{2\nu} y + \left(\sqrt{\frac{2}{\nu}}\right)\pi (n+\alpha) \right)
\end{align}
is a complete orthonormal system in $\Lsd$, where $z=x+iy\in \C$; $x,y\in\R$, and
$$ C_{m,n}^{\alpha,\nu} :=  ( 2^m m!)^{-1/2}  \left(\frac{2\nu}{\pi}\right)^{1/4}   e^{-\frac{\pi^2}{\nu}(\alpha+n)^2} .$$
\end{theorem}

\begin{proof}
Note for instance that one can check easily that the functions $\psi_{m,n}^{\alpha,\nu}$ satisfy the functional equation \eqref{FuncEq00}. Moreover, they form an orthonormal system in $\Lsd$. Indeed, if we let $\xi_{y,n}$ stands for $\xi_{y,n}= \sqrt{2\nu} y + \left(\sqrt{{2}/{\nu}}\right)\pi (n+\alpha)$,
then direct computation, yields
 \begin{align*}
 \scalnu{\psi_{j,k}^{\alpha,\nu},\psi_{m,n}^{\alpha,\nu}}
 & =  C_{j,k}^{\alpha,\nu}  \overline{C_{m,n}^{\alpha,\nu}} \int_S  e^{\frac{\nu}2 (z +\overline{z})^2 + 2i\pi [(\alpha + k) z + (\alpha + n) \overline{z}]} H_j (\xi_{y,k}) H_{m} (\xi_{y,n}) dxdy \\
& = C_{j,k}^{\alpha,\nu}  \overline{C_{m,n}^{\alpha,\nu}} \int_0^1 \int_{-\infty}^{+\infty} e^{2i\pi(k-n)x}   e^{-2\nu y^2 - 2\pi(2\alpha+k+n) y}
                                 H_j (\xi_{y,k})H_{m} (\xi_{y,n'}) dxdy \\
&= \delta_{k,n} C_{j,k}^{\alpha,\nu}  \overline{C_{m,n}^{\alpha,\nu}}  \left( \int_{-\infty}^{+\infty} e^{-2\nu y^2 - 4\pi(\alpha+k) y}
                              H_j (\xi_{y,n})H_{m} (\xi_{y,n}) dy\right).
\end{align*}
This follows since $\int_0^1 \int_{-\infty}^{+\infty} e^{2i\pi(k-n)x} dx = \delta_{k,n}$.
The change of variable $\xi=\xi_{y,n}$, so that $y=  \frac{\xi}{\sqrt{2\nu}} -  \frac{\pi (n+\alpha)}{\nu}$, leads to
\begin{align*}
\scalnu{\psi_{j,k}^{\alpha,\nu},\psi_{m,n}^{\alpha,\nu}}
       &= C_{j,k}^{\alpha,\nu} \overline{ C_{m,n}^{\alpha,\nu}} \left(\frac{1}{2\nu}\right)^{1/2}  \delta_{k,n}
             \int_{-\infty}^{+\infty}  e^{-\xi^2 + \frac{2\pi^2}{\nu}(\alpha+n)^2} H_j (\xi)   H_{m} (\xi) d\xi.
 \end{align*}
The last integral expresses the orthogonality of the Hermite polynomials, so that we have
\begin{align*}
\scalnu{\psi_{j,k}^{\alpha,\nu},\psi_{m,n}^{\alpha,\nu}}
&= |C_{m,n}^{\alpha,\nu}|^2 \left(\frac{1}{2\nu}\right)^{1/2} \delta_{k,n} \left(2^m\sqrt{\pi} m! \delta_{j,m}  e^{\frac{2\pi^2}{\nu}(\alpha+n)^2} \right)
= \delta_{k,n} \delta_{j,m}  .
 \end{align*}
In the other hand, note that the assertion of Proposition \ref{Prop-chara} holds true for $\Lsd$ and reads:
$f$ belongs to $\Lsd$ if and only if
    $ f(z):= e^{\frac{\nu}2 z^2 + 2i\pi \alpha z} g(z)$
     for certain almost everywhere simply periodic function $g$ (in the $x$-direction) and such that $\normnua{g}<+\infty$.
Thus, every $f\in \Lsd$ can be written as $$f(z) = e^{\frac\nu 2 z^2 + 2i\pi \alpha z}\sum\limits_{n\in\Z} a_n(y) e^{2i\pi n x} $$ in $\Lsd$.
Therefore, we have
$$ \normnu{f} = \sum\limits_{n\in\Z}   \left(\int_{-\infty}^{+\infty}|a_k(y)|^2 e^{-2\nu y^2 - 4\pi \alpha  y} dy \right) < +\infty  ,$$
whose the proof is similar to the one given for iii) in Theorem \ref{ThmbasisO}. In particular for every fixed integer $k$, we have
$$\int_{-\infty}^{+\infty}|a_n(y)|^2 e^{-2\nu y^2 - 4\pi \alpha  y} dy   < +\infty ,$$
which shows that the function
\begin{align}\label{Fct-b}
\xi \in \R \longmapsto b_n(\xi) := a_n\left( \frac{\xi}{\sqrt{2\nu}} -  \frac{\pi (\alpha+n)}{\nu}\right)  e^{\psi(\xi)} =
a_n(y)  e^{2\pi n y} e^{\frac{2\pi^2 (\alpha+n)^2}{\nu}}
\end{align}
belongs to $L^2(\R; e^{-\xi^2}d\xi)$, where $\psi(\xi)$ is the real valued function defined by $$\psi(\xi):= 2\pi n \left( \frac{\xi}{\sqrt{2\nu}}
-  \frac{\pi (\alpha+n)}{\nu}\right) +  \frac{2\pi^2 (\alpha+n)^2}{\nu}.$$
Furthermore, we have
\begin{align*}
\scalnu{f,\psi_{m,n}^{\alpha,\nu}}
&=C_{m,n}^{\alpha,\nu} \sum_{k\in\Z} \left( \int_0^1 e^{2i\pi (k-n) x} dx\right)
      \left( \int_{-\infty}^{+\infty} a_k(y) H_m \left(\xi_{y,n} \right) e^{-2\nu y^2 - 2\pi (2\alpha +n) y} dy\right) \nonumber
      \\
&=C_{m,n}^{\alpha,\nu}
     \int_{-\infty}^{+\infty} a_n(y) H_m \left(\xi_{y,n} \right)
     e^{-2\nu y^2 - 2\pi (2\alpha +n) y} dy   \nonumber\\
 &=  C_{m,n}^{\alpha,\nu} \left(\frac{1}{2\nu}\right)^{1/2}
  \int_{-\infty}^{+\infty} a_n\left( \frac{\xi}{\sqrt{2\nu}} -  \frac{\pi (n+\alpha)}{\nu}\right)  e^{\psi(\xi)}
   H_m \left(\xi \right) e^{-\xi^2} d\xi  
      \\
       &=  C_{m,n}^{\alpha,\nu} \left(\frac{1}{2\nu}\right)^{1/2}
  \int_{-\infty}^{+\infty} b_n(\xi)  H_m \left(\xi \right) e^{-\xi^2} d\xi.
\end{align*}
 Thus, to show that $(\psi_{m,n}^{\alpha,\nu})_{m,n}$ is a complete system in $\Lsd$, let assume that $\scalnu{f,\psi_{m,n}^{\alpha,\nu}} =0$ for every $m\in \Z^+$ and $n\in\Z$. Hence, it follows
$$
\int_{-\infty}^{+\infty} b_n(\xi) H_m \left(\xi \right) e^{-\xi^2} d\xi = 0
   $$
for every $m\in \Z^+$ and $n\in\Z$. By completeness of the Hermite polynomials $(H_m(\xi))_m$ in the Hilbert space $L^2(\R, e^{-\xi^2} d\xi)$, we conclude that the function
$b_n(\xi)$, defined by \eqref{Fct-b} and belonging to $ L^2(\R, e^{-\xi^2} d\xi)$, is identically zero on $\R$. This implies that $a_n\equiv 0$ on $\R$ for every $n\in \Z$ and therefore $f=0$. This completes de proof.
\end{proof}

\begin{corollary}
The space  $\Osd$ is a closed subspace of $\Lsd$. Its orthogonal complement in $\Lsd$ is spanned by $\psi_{m,n}^{\alpha,\nu}$ with $m=1,2,\cdots,$ and $n\in\Z$, in the sense that $f\in (\Osd)^{\perp}$ if and only if it is of the form
$$ f(z)= \sum_{m=1}^\infty \sum_{n\in\Z} a_{m,n} \psi_{m,n}^{\alpha,\nu}(z,\overline{z})
\quad \mbox{with} \quad
  \sum_{m=1}^\infty \sum_{n\in\Z} |a_{m,n}|^2 < +\infty .$$
\end{corollary}

Description of the Hilbert space $\Lsd$ can be investigated from a spectral theory point of view. In fact
this can be handled by introducing the first order differential operators $ A= {\partial}/{\partial \bar z}$ and its formal adjoint
$$A^*= -\frac{\partial}{\partial z} + \nu \overline{z}$$
with respect to the Gaussian density measure $ e^{-\nu|z|^2}\mes$, and next considering
 the elliptic self-adjoint differential operator
$\Delta_\nu 
=   A^* A = AA^* - \nu Id$
acting on $\Lsd$.

\begin{remark}
The shifted intertwined operator
$ e^{-\frac \nu 2|z|^2}\left(\Delta_\nu +\frac \nu 2\right)\left( e^{\frac \nu 2|z|^2} \cdot \right)
$ is the well known twisted Laplacian
$$
L_\nu := - \frac 14\left\{4\frac{\partial^2}{\partial z \partial \bar z} + 2\nu
\left(z\frac{\partial}{\partial z}-\bar z \frac{\partial}{\partial
\bar z}\right)- \nu^2 |z|^2\right\},
$$
interpreted as the Hamiltonian of a non-relativistic quantum
 particle moving on the $(x,y)$-plane under the action of an external constant magnetic field of magnitude $\nu$.
 It can also be linked to the sub-Laplacian
 on the Heisenberg group $H^{3}_{eis}=\C_z\times\R_t$ through the Fourier transform in $t$.
 \end{remark}

The associated eigenvalue problem
 \begin{align}\label{EigenVP}
\Delta^{\nu}_{\Gamma,\chi} f = \nu \lambda f; \quad f\in \Lsd,
\end{align}
can be reduced, by ellipticity of $\Delta_\nu$, to the space $\mathcal{C}^{\infty,\nu}_{L,\alpha}(\C)$ of smooth functions $f: \C
\stackrel{\mathcal{C}^\infty}{\longrightarrow} \C$  satisfying
$$ f(z+m) = e^{\nu\left( z + \frac m 2\right) m +2i\pi \alpha m} f(z); \qquad \forall\forall z\in \C, \, \forall m\in\Z.$$
A fundamental fact is that the operator $A^*$ and $\Delta_\nu$ preserve the quasi-periodic property. Furthermore, $\Delta_\nu$ leaves invariant the eigenspace
\begin{align*}
\Elambda &= \{f\in \Lsd \, \mbox{ such that } \Delta_{\nu} f = \nu \lambda f\} \\
&=  \{f\in \mathcal{C}^{\infty,\nu}_{L,\alpha}(\C) \, \mbox{ such that } \Delta_{\nu} f = \nu \lambda f \, \mbox{ and } \, \normnu{f} <+\infty  \}.
\end{align*}
%
%
One shows also that $\Elambda$ is non-trivial if and only if $\lambda =m$; for certain $m=0,1,2, \cdots$.
with $\Ell=\Osd$. Thence, the following Hilbertian orthogonal decomposition
\begin{align}
\Lsd &= \Osd \oplus \left(\bigoplus_{m=1}^\infty \Em \right)
\end{align}
holds. Moreover, one obtains that for fixed positive integer $m=1,2, \cdots$, the operator $(\nu^m m!)^{-1/2}[i A^*]^m $ maps isometrically
 the orthogonormal basis $\psin$ of $\Osd$ (given through \eqref{orthonorm-basis}) to an orthonomal basis of the $L^2$-eigenspace
 $\Em$. More precisely, we have
\begin{align}
[(\nu^m m!)^{-1/2} [i A^*]^m \psin](z) = \psi_{m,n}^{\alpha,\nu}(z,\overline{z}),
\end{align}
where $ \psi_{m,n}^{\alpha,\nu}(z,\overline{z})$ are those in \eqref{CloHer}. The previous result follows easily by induction on $m$ using the fact that
$[A^*]^{m+1} \psin  =  A^* (\psi_{m,n}^{\alpha,\nu}) =\nu \overline{z}\psi_{m,n}^{\alpha,\nu}     - \pz  \psi_{m,n}^{\alpha,\nu}$ combined with
 the well-known recurrence formula $H_{m+1}(x) = 2 x H_{m}(x) - H_{m}'(x)$ for the Hermite polynomials.



\end{document}